\documentclass{amsart}
\usepackage{amsxtra,amscd}
\usepackage{graphicx}
\usepackage{amsmath}
\usepackage{amsfonts}
\usepackage{amssymb}

\newtheorem{theorem}{Theorem}[section]
\newtheorem{lemma}[theorem]{Lemma}
\theoremstyle{definition}
\newtheorem{definition}[theorem]{Definition}
\newtheorem{example}[theorem]{Example}

\newtheorem{proposition}[theorem]{Proposition}
\newtheorem{remark}[theorem]{Remark}

\numberwithin{equation}{section}

\begin{document}
\title{On the \'{e}tale fundamental groups of arithmetic schemes}
\author{Feng-Wen An}
\address{School of Mathematics and Statistics, Wuhan University, Wuhan,
Hubei 430072, People's Republic of China}
\email{fwan@amss.ac.cn}
\subjclass[2010]{Primary 14F35; Secondary 11G35}
\keywords{arithmetic scheme, arithmetically unramified, \'{e}tale fundamental group,  formally
unramified}

\begin{abstract}
In this paper we will give a computation of the \'{e}tale fundamental group
of an integral arithmetic scheme. For such a scheme, we will prove that the
\'{e}tale fundamental group is naturally isomorphic to the Galois group of
the maximal formally unramified extension over the function field. It consists of
the main theorem of the paper. Here, formally unramified will be proved to
be arithmetically unramified which is defined in an evident manner and coincides with that in algebraic number theory.
Hence, formally unramified has an arithmetic sense.
At the same time, such a computation coincides
with the known result for a normal  noetherian scheme.
\end{abstract}

\maketitle

\begin{center}
{\tiny {Contents} }
\end{center}

{\tiny \qquad {Introduction} }

{\tiny \qquad {1. Statement of The Main Theorem} }

{\tiny \qquad {2. Preliminaries}}

{\tiny \qquad {3. A Criterion of Galois Covers}}

{\tiny \qquad {4. The Calculus and Construction of Galois Covers}}

{\tiny \qquad {5. Formally Unramified = Arithmetically Unramified}}

{\tiny \qquad {6. Proof of The Main Theorem} }

{\tiny \qquad {References}}{}

\section*{Introduction}

In this paper we will give a computation of the \'{e}tale fundamental group
of an arithmetic variety, i.e., an integral arithmetic scheme surjectively over $Spec(\mathbb{Z})$ of finite type.

Let $X$ be an arithmetic variety. It will be
proved that the \'{e}tale fundamental group $\pi ^{et}_{1}(X,s)$ is isomorphic to the Galois group
$Gal(k(X)^{un}/k(X))$ of the maximal formally
unramified extension $k(X)^{un}$ of the function field $k(X)$ in a natural manner. See the
main theorem of the paper.

Such a computation of the \'{e}tale fundamental group
coincides with the known result on a normal  noetherian
scheme (for instance, see \cite{f-k,sga1,mln,sz}).

In particular, we will prove that the maximal formally unramified extension $k(X)^{un}$
is equal to the maximal arithmetically
unramified extension $k(X)^{au}$.

Here, arithmetically
unramified, which is defined in an evident manner (see \S 5), coincides with that in algebraic number theory and has many arithmetic operations and properties.
Hence, the maximal formally unramified extension is exactly of an arithmetic
sense.

At the same time, the arithmetically unramified is more accessible to operations. Naturally,
the arithmetically unramified
can serve as a computation of the \'{e}tale fundamental group of
an arithmetic variety.

For instance, as an immediate application, take a finite number of variables $t_{1},t_{2},\cdots, t_{n}$
 over $\mathbb{Q}$. It is proved that $\mathbb{Q}(t_{1},t_{2},\cdots, t_{n})$ has a trivial maximal arithmetically
 unramified extension (see \cite{an6}). Therefore, the arithmetic variety $Spec(\mathbb{Z}[t_{1},t_{2},\cdots, t_{n}]$ has a trivial
 \'{e}tale fundamental group (see \emph{Example 1.8}), i.e.,
$$
\pi _{1}^{et}\left( Spec(\mathbb{Z}[t_{1},t_{2},\cdots, t_{n}])\right)\cong \{0\}.
$$

\section{Statement of the Main Theorem}

\subsection{Notations}

In this paper, an \textbf{arithmetic variety} is an integral scheme that is
surjectively over $Spec\left( \mathbb{Z}\right) $ of finite type.

For an integral scheme $Y$, let $k(Y)= \mathcal{O}_{Y,\xi}$ be the
function field of $Y$, where $\xi$ is the generic point of $Y$. Denote by
$
\pi _{1}^{et}\left( Y,s\right)
$ the \'{e}tale fundamental group of $Y$ over a geometric point $s$ of $Y$ over a given algebraic closure of the function field $k(Y)$.
Sometimes, we write $\pi _{1}^{et}\left( Y\right)=\pi _{1}^{et}\left( Y,s\right)$ for brevity.

Let $Gal(L/K)$ denote the Galois group
of an extension $L$ of a field $K$.

\subsection{Formally unramified}

Fixed an arithmetic variety $X$. Let $\Omega$ be  an algebraic closure  of
the function field ${k(X)}$ and $s_{\xi}$ the geometric point  of $X$ over $\Omega$ such that $s_{\xi}(Spec(\Omega))$ is the generic point $\xi$ of $X$.

Put
\begin{itemize}
\item $\left[X; \Omega \right]_{et}\triangleq$
the set of the equivalence classes of  pointed covers $(Y,s_{Y})$ of $(X,s_{\xi})$ such that $Y$ is an irreducible finite
\'{e}tale Galois cover of $X$ and  $s_{Y} $ is a geometric point of $Y$ over $s_{\xi}$. Here,
two such covers $(Y,s_{Y}),(Z,s_{Z}) $ of $(X,s_{\xi})$ are
equivalent if there is an isomorphism $f$ of $Y$ onto $Z$
over $X$ with $f(s_{Y})=s_{Z}$.
\end{itemize}

For finite \'{e}tale Galois covers, see \S 6.1 below in the paper or see \cite{f-k,sga1,mln,sz}.

For brevity, denote by $Y$ or $(Y,s_{Y})$ for an equivalence class of the pointed cover $(Y,s_{Y})$ that
is contained in $ \left[
X;\Omega \right] _{et}$.

For any $X_{\alpha},X_{\beta}\in \left[X; \Omega \right]_{et}$, we say $$
X_{\alpha}\leq X_{\beta}$$ if and only if $X_{\beta}$ is a finite \'{e}tale
Galois cover over $X_{\alpha}.$ Then $\left[X; \Omega \right]_{et} $ is a
directed set with the partial order $\leq$.

Set
\begin{itemize}
\item $
k(X; \Omega)^{un}\triangleq {{\lim}_{\rightarrow_{Z\in {\left[X; \Omega \right]_{et}}}}}{\lambda_{Z}(k(Z))}
$, i.e., $k(X; \Omega)^{un}$ is the direct limit of the direct system of fields $k(Z)$
indexed by $\left[X; \Omega \right]_{et}$, where each $\lambda_{Z}:k(Z)\to
\Omega$ is the $k(X)$-embedding of fields.
\end{itemize}

The subfield $k(X; \Omega)^{un}$ of $\Omega$ is said to be the \textbf{maximal
formally unramified extension} of the arithmetic variety $X$ (or of the function field $k(X)$) in $\Omega$.

It is seen that $k(X; \Omega)^{un}$ is well-defined since all the $k(X)$-embeddings $$
\lambda_{Z}:k(Z)\to \Omega$$ are compatible in a natural manner. Moreover, $k(X; \Omega)^{un}$
is an algebraic Galois extension of the field $k(X)$.

In deed, it will also be seen that the maximal formally unramified extension
is exactly of an arithmetic sense that coincide with algebraic number theory (see \S\S 1.4, 5.2 for details;
see \cite{an6,an7} for further results).

\subsection{Statement of the main theorem}

Now we give a computation of the \'{e}tale fundamental group of an
arithmetic variety, which is the main theorem of the paper.

\begin{theorem}
For an arithmetic variety $X$, the \'{e}tale fundamental group $\pi _{1}^{et}\left( X,s\right)$ is naturally
isomorphic to the Galois group of the maximal arithmetically unramified extension $k(X; \Omega)^{un}$
over the function field $k(X)$.

That is, there is an isomorphism
\begin{equation*}
\pi _{1}^{et}\left( X,s\right) \cong Gal\left( k(X; \Omega)^{un}/k\left( X\right)
\right)
\end{equation*}
between groups in a natural manner.

Here, $s$ is a geometric point of $X$ over an algebraic closure $\Omega$ of the function field $k(X)$.
\end{theorem}

\begin{remark}
Such a computation above coincides with the known results on \'{e}tale
fundamental groups of normal connected schemes (e.g., see \cite{an4,f-k,sga1,mln,sz}).
\end{remark}

We will prove \emph{Theorem 1.1} in \S 6 after we make preparations in \S \S
2-5.

\subsection{Concluding Remarks}

There are the following remarks as a conclusion of the section.

\begin{remark}
The \'{e}tale fundamental groups of arithmetic varieties are a birational invariant.
In deed, let $X$ and $Y$ be two arithmetic varieties such that the function fields $k(X)\cong k(Y)$ are isomorphic. Then there is a natural isomorphism
$$\pi _{1}^{et}\left( X\right)\cong \pi _{1}^{et}\left( Y\right)$$ between the \'{e}tale fundamental groups.
\end{remark}

\begin{remark}
Fixed an arithmetic variety $X$ and an algebraic closure $\Omega$ of $k(X)$.
Let $k(X,\Omega)^{au}$ be the
the maximal arithmetically unramified extension of $k(X)$
in $\Omega$ (see \S 5 for detail). It is seen that $$k(X; \Omega)^{au}=k(X; \Omega)^{un}$$ holds,
i.e., for the field $k(X)$, the maximal formally unramified extension is equal to the
maximal arithmetically unramified extension that is given in an arithmetic manner  (see \emph{Lemma 5.9}). Hence, the formally unramified extension
$k(X; \Omega)^{un}$ has a sense of arithmetics.
\end{remark}

\begin{remark}
In \cite{an2b} we also give a computation of the \'{e}tale fundamental group of an algebraic scheme by virtue of
algebraic unramified extensions over the function field. In deed, let $Z$ be an integral algebraic scheme $Z$ over a field $K$ of finite type.
It is seen that the \'{e}tale fundamental group $\pi_{1}^{et}(Z,s)$
is naturally isomorphic to the Galois group of
the maximal algebraically unramified extension over the function field $k(Z)$. However, the algebraically unramified and the arithmetically unramified are very different
from each other. In particular, the algebraically unramified is different from that in algebraic number theory.
\end{remark}

\begin{remark}
Let $X$ be an arithmetic variety and
let $\mathcal{O}_{k(X)}$ be the ring of algebraic integers in the field $k(X)$ (see \S 5 for detail).
There is an isomorphism
$$\pi _{1}^{et}\left( X\right)\cong \pi _{1}^{et}\left(Spec(\mathcal{O}_{k(X)})\right)$$
in a natural manner. Hence, to compute $\pi _{1}^{et}\left( X\right)$, canonically it reduces  to find the \'{e}tale
fundamental group of the affine scheme $Spec(\mathcal{O}_{k(X)})$.
\end{remark}

\begin{remark}
Let $X$ be an arithmetic variety. Put $L=k(X)$. Then there are isomorphisms $$\pi _{1}^{et}\left( X\right)\cong
Gal(L^{au}/L)\cong Gal(L^{un}/L)$$ of groups in canonical
manner from \emph{Theorem 1.1, Remarks 1.4, 5.10}.
\end{remark}

At the end of the section, as an application of the main theorem in the paper, we give a computation of the \'{e}tale fundamental group of the ring of algebraic integers in a purely
transcendental extension over $\mathbb{Q}$.

\begin{example}
Let $t_{1},t_{2},\cdots, t_{n}$ be variables over $\mathbb{Q}$.
Then
$$
\pi _{1}^{et}\left( Spec(\mathbb{Z}[t_{1},t_{2},\cdots, t_{n}])\right)\cong \{0\}
$$
since
$$\mathbb{Q}(t_{1},t_{2},\cdots, t_{n})^{au}=\mathbb{Q}(t_{1},t_{2},\cdots, t_{n})$$ holds (see \cite{an6}).
\end{example}

\section{Preliminaries}

In the following we will fix notations and terminology which will be used
to prove the main theorem of the paper.

\subsection{Notations}

In the remainder of the paper, by an \textbf{integral variety} we will understand a finite-dimensional integral
scheme that is surjectively over $Spec(\mathbb{Z})$.

Define
$L^{al}$ (or $\overline{L}$) $\triangleq $ an algebraical closure of a field $L $;
$Fr(D)\triangleq $ the field of fractions of an integral domain $D$.

Let $f:X\rightarrow Y$ be a morphism of integral schemes. Denote by $f^{\sharp}:k(Y)\to k(X)$ the homomorphism of the function fields given by $f$; put
$Aut\left( X/Y\right) \triangleq $ the group of automorphisms of $X$
over $Y$.

Let $A$ be a commutative ring. Set
\begin{itemize}
\item $j_{x}(A)$ (or $j_{x}$) $\triangleq $ the prime ideal of $A$
corresponding to a point $x\in Spec(A)$.
\end{itemize}

\subsection{Essentially affine schemes}

Let $X$ be a scheme. An \textbf{affine covering} of $X$ is a family
\begin{equation*}
\mathcal{C}_{X}=\{(U_{\alpha },\phi _{\alpha };A_{\alpha })\}_{\alpha \in
\Delta }
\end{equation*}
where for each $\alpha \in \Delta $, $\phi _{\alpha }$ is an isomorphism
from an open set $U_{\alpha }$ of $X$ onto the spectrum $Spec{\ A_{\alpha }}$
of a commutative ring $A_{\alpha }$. Each $(U_{\alpha },\phi _{\alpha
};A_{\alpha })\in \mathcal{C}_{X}$ is called a \textbf{local chart}.
For  brevity, a local chart $(U_{\alpha},\phi_{\alpha};A_{\alpha
})$ will sometimes be denoted by $U_{\alpha}$ or $(U_{\alpha},\phi_{\alpha})$.

An affine covering $\mathcal{C}_{X}$ of $(X, \mathcal{O}_{X})$ is said to be
\textbf{reduced} if $U_{\alpha}\neq U_{\beta} $ holds for any $\alpha\neq
\beta$ in $\Delta$.

Let $\mathfrak{Comm}$ be the category of commutative rings with identity.
For a given field $\Omega$, let $\mathfrak{Comm}(\Omega)$ be the category
consisting of the subrings of $\Omega$ and their isomorphisms.

\begin{definition}
Let $\mathfrak{Comm}_{0}$ be a subcategory of $\mathfrak{Comm}$. An affine
covering of $X$, namely $\{(U_{\alpha},\phi_{\alpha};A_{\alpha })\}_{\alpha
\in \Delta}$, is said to be \textbf{with values} in $\mathfrak{Comm}_{0}$ if
for each $\alpha \in \Delta$, there are the properties:
$$
\mathcal{O}_{X}(U_{\alpha})=A_{\alpha};$$
$$ U_{\alpha}=Spec(A_{\alpha}),
$$
where each $A_{\alpha }$ is a ring contained in $\mathfrak{Comm}_{0}$.

In particular, an affine covering $\mathcal{C}_{X}$ of $X$ with values in $%
\mathfrak{Comm}(\Omega)$ is said to be \textbf{with values in the field $%
\Omega$}.
\end{definition}

\begin{definition}
An integral scheme $X$ is said to be \textbf{essentially affine} in a field $\Omega $ if for every affine open set $U$ of $X$
there are the properties: $$\mathcal{O}_{X}(U)\subseteq \Omega ;$$ $$U=Spec(\mathcal{O}
_{X}(U)).$$
\end{definition}

Here we give such a fact that every integral variety $X$ is isomorphic to an
integral variety $Z$ that is essentially affine in the field $k(X)$.

\begin{lemma}
For any integral variety $X$, there is an integral variety $Z$ satisfying the
properties:
\begin{itemize}
\item $k\left( X\right) =k\left( Z\right);$

\item $X\cong Z$ are isomorphic schemes;

\item $Z$ is essentially affine in the field $k(Z)$.
\end{itemize}
\end{lemma}

Such an integral variety $Z$ is called an \textbf{essentially affine realization} of $X$.

\begin{proof}
Fixed an affine open set $U$ of $Y$. Let
\begin{equation*}
\phi_{U}: (U, \mathcal{O}_{X}|_{U})\to (Spec(A_{U}), \mathcal{O}
_{Spec(A_{U})})
\end{equation*}
be an isomorphism of schemes. We have ring isomorphisms
\begin{equation*}
\phi_{U} ^{\sharp}:A_{U}\to \mathcal{O}_{X}(U);
\end{equation*}
\begin{equation*}
i_{X}:\mathcal{O}_{X}(U)\to B_{U}\subseteq k(X).
\end{equation*}
The composite of ring isomorphisms
\begin{equation*}
t_{U}= i_{X}\circ \phi_{U} ^{\sharp}:A_{U}\to B_{U}
\end{equation*}
induces an isomorphism
\begin{equation*}
\tau_{U}:(Spec{( A_{U})},\mathcal{O}_{Spec{( A_{U}})})\to (Spec{( B_{U})},
\mathcal{O}_{Spec( B_{U})})
\end{equation*}
between affine schemes.
Put
\begin{equation*}
\psi_{U}=\tau _{U}\circ \phi_{U}.
\end{equation*}

Then there is a new scheme $Z$ obtained by gluing all such affine schemes
\begin{equation*}
(Spec( B_{U}),\mathcal{O}_{Spec( B_{U})})
\end{equation*}
along $\psi_{U}$ for every affine open set $U$ of $X$ (see \cite{ega,hartshorne,iitaka}).
It is seen that the scheme $Z$ has the desired properties.
\end{proof}

\subsection{Affine points in a field}

 The affine points in a field will be applied  to a universal construction of a Galois cover of an
integral variety in \S 4.
Let $\Psi$ be an arbitrary field.

\begin{definition}
By
an \textbf{affine $\Psi$-point} (or \textbf{affine point} in $\Psi$) we understand a point contained in the underlying space
of an affine scheme $Spec(A)$ for a subring $A$ of the field $ \Psi$ such that $Fr(A)=\Psi$.

Denote by $[\Psi]_{ap}$ the set of all affine points in the field $\Psi$.
\end{definition}

In general, it is not true that there is an affine scheme that has $[\Psi]_{ap}$ as the underlying space.

\begin{definition}
Let $x,y\in [\Psi]_{ap}$ be two affine points.

(i) $x$ and $y$ are  \textbf{$\Psi$-equal} if there
is a subring $A$ of the field $ \Psi$ such that
\begin{itemize}
\item $Fr(A)=\Psi$;

\item $x,y \in Spec(A)$;

\item $j_{x}(A)=j_{y}(A)$ holds as prime ideals in $A$.
\end{itemize}

(ii) $x$ and $y$ are \textbf{$\Psi$-equivalent}
if there are a tower of subrings $A\subseteq B\subseteq \Psi$ satisfying the properties:
\begin{itemize}
\item $Fr(A)=Fr(B)=\Psi$;

\item $x\in Spec(A), \, y\in Spec(B)$;

\item $j_{x}(A)=\rho^{-1}_{A,B}(j_{y}(B))$ holds as prime ideals in $A$,
\end{itemize}
where $\rho_{A,B}:A \hookrightarrow B$ is the inclusion map.
\end{definition}

It is clear that affine $\Psi $-points that are $\Psi$-equal are $\Psi $-equivalent.

\subsection{Function fields and invariant subrings}

Let $X$ be an integral variety with generic point $\xi $. Recall that for the function field $k(X)$,
 there is a natural inclusion
\begin{equation*}
i_{X}:\mathcal{O}_{X}(U)\rightarrow k(X)\triangleq \mathcal{O}_{X,\xi }
\end{equation*}
of rings for each open set $U$ in $X$, i.e.,
\begin{equation*}
k(X)=\{i_{X}(x)\in k(X):x\in \mathcal{O}_{X}(U),U\text{ is open in }X\}.
\end{equation*}

Furthermore, $k(X)$ is isomorphic onto the field
\begin{equation*}
Rat(X)\triangleq\{[U,f]:f\in \mathcal{O}_{X}(U),U\text{ is open in }X\},
\end{equation*}
where $[U,f]$ denotes the germ of $(U,f)$ (see \cite{ega,hartshorne,iitaka}).

For brevity, we will always identify $k(X)$, $(U,f)$, and $i_{X}(x)\in k(X)$
with $Rat(X)$, $[U,f]$, and $x\in \mathcal{O}_{X}(U)$, respectively.

Each automorphism $\sigma =(\sigma,\sigma ^{\sharp})$ of $X$ gives an automorphism
\begin{equation*}
\sigma ^{\sharp}\triangleq\sigma ^{\sharp}_{\xi}:\mathcal{O}_{X,\xi} \to \mathcal{O}_{X,\xi}
\end{equation*}
of the function field $k(X)$.

For an open set $U$ in $X$ and a subgroup $G\subseteq Aut(X/Y)$, put
\begin{itemize}
\item ${{\mathcal{O}_{X}}^{G}}(U)$ $\triangleq $ the subring of all
$x\in \mathcal{O}_{X}(U)$ such that $\sigma _{\xi }^{\sharp
}(i_{X}(x))=i_{X}(x)$ holds for every $\sigma \in
G $.
\end{itemize}

\subsection{Galois covers}

Let $f:X\rightarrow Y$ be a surjective morphism of integral schemes. Note that here $f$ is not necessarily of finite type.

\begin{definition}
$X$ is \textbf{Galois} over $Y$ by $f$ if the two conditions are satisfied:

\begin{itemize}
\item $k(X)$ is naturally  Galois  over $k(Y)$ by $f$;

\item $Aut(X/Y)\cong Gal(k(X)/k(Y))$ are isomorphic in a natural manner.
\end{itemize}
\end{definition}

Let $D\subseteq D_{1}\cap D_{2}$ be three integral domains contained in a
field $\Omega$. The ring $D_{1}$ is said to be a \textbf{conjugation} of $D_{2}$ over
$D$ if there is a $k$-isomorphism $$\tau:Fr(D_{1})\rightarrow Fr(D_{2})$$ between
fields such that $\tau(D_{1})=D_{2},$ where $k\triangleq Fr(D)$.

\begin{definition}
$X$ is \textbf{locally quasi-galois} over $Y$ by $f$ if the below condition
is satisfied:
\begin{quotation}
For any affine open subsets $V\subseteq Y$ and $U\subseteq f^{-1}(V)$, the
ring $\mathcal{O}_{X}(U)$ is contained in a common algebraic closure $\Omega$
 of $k(X)$ and there is one and only one conjugation of $\mathcal{%
\ O}_{X}(U)$ over the ring $f^{\sharp}(i_{Y}(\mathcal{O}_{Y}(V)))$.
\end{quotation}
\end{definition}

\begin{definition}
A reduced affine covering $\mathcal{C}_{X}$ of $X$ with values in an algebraic closure $\Omega$ of $k(X)$
is said to be \textbf{quasi-galois closed} over $Y$ by $f$ if there is a
local chart $(U_{\alpha }^{\prime },\phi _{\alpha }^{\prime };A_{\alpha
}^{\prime })\in \mathcal{C}_{X}$ such that $U_{\alpha }^{\prime }\subseteq
\varphi^{-1}(V_{\alpha})$ and $H_{\alpha } = A_{\alpha }^{\prime }$ hold for

\begin{itemize}
\item any $(U_{\alpha },\phi _{\alpha };A_{\alpha })\in \mathcal{C}_{X}$;

\item any affine open set $V_{\alpha}$ in $Y$ with $U_{\alpha }\subseteq
f^{-1}(V_{\alpha})$;

\item any conjugation $H_{\alpha }$ of $A_{\alpha }$ over $%
B_{\alpha}\triangleq f^{\sharp}(i_{Y}(\mathcal{O}_{ Y}(V_{\alpha})))$.
\end{itemize}
\end{definition}

\begin{definition}
$X$ is \textbf{quasi-galois closed} (or \textbf{\emph{qc}} for short) over $%
Y $ by $f$ if there is a reduced affine covering $\mathcal{C}_{X}$ of $X$
with values in an algebraic closure $\Omega$ of $k(X)$ such that $\mathcal{C}_{X}$ is quasi-galois
closed over $Y$ by $f$.
\end{definition}

\begin{definition}
$X$ is \textbf{locally complete} over $Y$ by $f$ if for every automorphism $\rho\in Gal(k(X)/f^{\sharp}(k(Y)))$ and every affine open sets $V\subseteq Y$ and $U\subseteq f^{-1}(V)$, there is an affine open set $W\subseteq f^{-1}(V)$ and an isomorphism $\widetilde{\rho}:\mathcal{O}_{X}(U)\to \mathcal{O}_{X}(W)$ between algebras over $f^{\sharp}(\mathcal{O}_{Y}(V))$ such that $\rho$ is induced from $\widetilde{\rho}$ in a natural manner.
\end{definition}

Immediately, by definitions we have
$$\emph{locally quasi-galois} \implies \emph{locally complete};$$
$$\emph{quasi-galois closed} \implies \emph{locally complete}.$$

On the other hand, from \emph{Step 4} in \S 3.3 it is seen that the above condition local completeness is automatically satisfied for integral affine schemes and for integral normal schemes.

\section{A Criterion of Galois Covers}

In this section we will give a criterion  of Galois covers of integral varieties,
which is one of the key points for us to compute the \'{e}tale
fundamental group of an arithmetic variety in the paper. In \S 4 by such a result
we will prove that there is a universal Galois cover for the \'{e}tale
fundamental group.

\subsection{A criterion of Galois covers}

The following theorem gives a criterion of Galois covers for integral varieties, which will be frequently used in the paper.

\begin{theorem}
Let $X$ and $Y$ be integral varieties. Suppose that $\phi:X\rightarrow Y$ is a surjective morphism of schemes such that $X$ is locally complete over $Y$ and $k\left( X\right) $ is
canonically an algebraic Galois extension over $k\left( Y\right) $ by $\phi$.
Then there is a natural
isomorphism
\begin{equation*}
Aut\left( X/Y\right) \cong Gal\left( k\left( X\right) /k\left( Y\right)
\right)
\end{equation*}
between groups.
In particular, $X$ is a Galois cover of $Y$.
\end{theorem}

\begin{remark}
The conclusion of \emph{Theorem 3.1} can be regarded as a generalization of
several well-known related results in \cite{f-k,sga1,sv1,sv2,sz} for the case of integral varieties since the structure
morphism $\phi$ here is not necessarily of finite type.
\end{remark}

\subsection{Applications of the criterion}

Let $X$ and $Y$ be integral varieties of the same dimension. Let $f :X\rightarrow Y$ be a surjective morphism.

\begin{lemma}
Assume that $X$ is locally
quasi-galois over $Y$ by $f$. Then $X$ is $qc$ over $Y$.
\end{lemma}

\begin{proof}
It is immediate from definition.
\end{proof}

\begin{lemma}
Let $X$ be $qc$ over $Y$ by
$f$. Then the function field $k\left( X\right) $ is canonically Galois over $f(Y)$; moreover, $X$ is Galois over $Y$ by $f$.
\end{lemma}

\begin{proof}
It is immediate from \emph{Theorem 3.1}.
\end{proof}

\begin{lemma}
Suppose that for any affine
open set $V$ of $Y$, the set $f^{-1}(V)$ is affine open in $X$ and the ring
$f^{\sharp}(i_{Y}(\mathcal{O}_{Y}(V)))$ is equal to $i_{X}({{\mathcal{O}_{X}}
^{Aut(X/Y)}(f^{-1}(V))})$. Then $X$ is Galois over $Y$ by $f$.
\end{lemma}

\begin{proof}
It is immediate from \emph{Lemmas 3.4-5}.
\end{proof}

\emph{Lemmas 3.3-5} will be used in the following section.

\subsection{Proof of Theorem 3.1}

Now we prove \emph{Theorem 3.1}.

\begin{proof}
\textbf{(Proof of Theorem 3.1)} The process of the proof here is based on a
trick originally in \cite{an} and a refinement in \cite{an5}.

 Without loss of generality, by \emph{Lemma
2.3} suppose that $X$ and $Y$ are essentially affine in the function fields $k(X)$ and $k(Y)$,
respectively.

Here, the function field $k(X)$ will be taken as the set of elements
of the form ${\left( U,f\right)}$. We will
identify the function field $k(Y)$ with the subfield $\phi^{\sharp}\left(k\left( Y\right)\right)$
of $k(X)$.

In the following we will proceed in several steps to prove that there is a
group isomorphism
$$
Aut\left( X/Y\right) \cong Gal\left( k\left( X\right) /k\left(
Y\right) \right).
$$

\textbf{Step 1.} Construct a map
\begin{equation*}
t:Aut\left( X/Y\right) \longrightarrow Gal\left( k\left( X\right) /k\left(
Y\right) \right)
\end{equation*}
between sets.

In fact, given an automorphism
\begin{equation*}
\sigma =\left( \sigma ,\sigma ^{\sharp}\right) \in Aut_{k}\left( X/Y\right) ,
\end{equation*}
i.e.,
\begin{equation*}
\sigma : X \longrightarrow X
\end{equation*}
is a homeomorphism;
\begin{equation*}
\sigma ^ {\sharp}:\mathcal{O}_{X} \rightarrow \sigma _{\ast }\mathcal{O}_{X}
\end{equation*}
is an isomorphism of sheaves of rings on $X$.

Let $\xi$ be the generic point of $X$. We have
  $$\sigma (\xi)=\xi;$$ it follows that
\begin{equation*}
\sigma ^{\sharp}:k\left( X\right)=\mathcal{O}_{X,\xi } \rightarrow \sigma
_{\ast }\mathcal{O}_{X,\xi }=k\left( X\right)
\end{equation*}
is an automorphism of $k(X)$.

Denote by $\sigma ^{\sharp-1}$ the inverse of the ring isomorphism
\begin{equation*}
\sigma ^{\sharp}:k(X)\to k(X).
\end{equation*}

Take any open subset $U$ of $X$. We have the restriction
\begin{equation*}
\sigma=(\sigma ,\sigma ^{\sharp}): (U,\mathcal{O}_{X}|_{U}) \longrightarrow
(\sigma(U),\mathcal{O}_{X}|_{\sigma(U)})
\end{equation*}
of open subschemes. That is,
\begin{equation*}
\sigma^{\sharp}:\mathcal{O}_{X}|_{\sigma(U)} \rightarrow \sigma_{\ast}
\mathcal{O}_{X}|_{U}
\end{equation*}
is an isomorphism of sheaves on $\sigma (U)$. In particular,
\begin{equation*}
\sigma^{\sharp}:\mathcal{O}_{X}(\sigma(U)) =\mathcal{O}_{X}|_{\sigma(U)}(
\sigma(U)) \rightarrow \mathcal{O}_{X}(U)=\sigma_{\ast} \mathcal{O}
_{X}|_{U}(\sigma(U))
\end{equation*}
is an isomorphism of rings.

For every $f \in \mathcal{O}_{X}|_{U}(U)$, we have
\begin{equation*}
f \in \sigma_{\ast}\mathcal{O}_{X}|_{U}(\sigma(U));
\end{equation*}
hence
\begin{equation*}
\sigma^{\sharp -1}(f) \in \mathcal{O}_{X}(\sigma(U)).
\end{equation*}

Now we have a map
\begin{equation*}
t:Aut_{k}\left( X/Y\right) \longrightarrow Gal( k\left( X\right)
/k( Y ))
\end{equation*}
between sets given by
\begin{equation*}
\sigma =(\sigma ,\sigma ^{\sharp})\longmapsto t(\sigma)=\left\langle \sigma
,\sigma ^{\sharp -1}\right\rangle
\end{equation*}
such that
\begin{equation*}
\left\langle \sigma ,\sigma ^{\sharp-1}\right\rangle :\left( U,f\right)\in
\mathcal{O}_{X}(U) \longmapsto \left( \sigma \left( U\right) ,\sigma
^{\sharp-1}\left( f\right) \right)\in \mathcal{O}_{X}(\sigma(U))
\end{equation*}
is a mapping of $k(X)$ into $k(X)$.

\textbf{Step 2.} Prove that the map
\begin{equation*}
t:Aut\left( X/Y\right) \longrightarrow Gal\left( k\left( X\right) /k\left(
Y\right) \right)
\end{equation*}
is a homomorphism between groups.

In deed, given any
\begin{equation*}
\sigma =\left( \sigma ,\sigma ^{\sharp}\right) \in Aut\left( X/Y\right).
\end{equation*}
For any $(U,f),(V,g) \in k(X)$, we have
\begin{equation*}
(U,f)\cdot(V,g)=(U\cap V, f\cdot g);
\end{equation*}
\begin{equation*}
(U,f)+(V,g)=(U\cap V, f+g).
\end{equation*}
Then
$$
\left\langle \sigma ,\sigma ^{\sharp-1}\right\rangle((U,f)\cdot(V,g)) \\
=\left\langle \sigma ,\sigma ^{\sharp-1}\right\rangle((U,f))\cdot
\left\langle \sigma ,\sigma ^{\sharp-1}\right\rangle((V,g));
$$
$$
\left\langle \sigma ,\sigma ^{\sharp-1}\right\rangle((U,f)+(V,g))
=\left\langle \sigma ,\sigma ^{\sharp-1}\right\rangle((U,f))+ \left\langle
\sigma ,\sigma ^{\sharp-1}\right\rangle((V,g)).
$$
Hence, $\left\langle \sigma ,\sigma ^{\sharp-1}\right\rangle$ is an
automorphism of $k\left( X\right) .$

It needs to prove that $\left\langle \sigma ,\sigma ^{\sharp-1}\right\rangle$
is an automorphism over $k(Y)$. Consider the given structure morphism
\begin{equation*}
\phi=(\phi,\phi^{\sharp}):(X,\mathcal{O} _{X})\rightarrow (Y,\mathcal{O}_{Y})
\end{equation*}
between schemes. As $\phi(\xi)$ is the generic point of $Y$ and $\xi$
is invariant under any automorphism $\sigma \in Aut\left( X/Y\right)$, it is seen that $
\sigma^{\sharp}: \mathcal{O}_{X,\xi} \rightarrow \mathcal{O}_{X,\xi}$ is an
isomorphism of algebras over the subalgebra $$k(Y)=
\phi^{\sharp}(k(Y))=\phi^{\sharp}(\mathcal{O}_{Y,\phi(\xi)}),$$ i.e.,
\begin{equation*}
\left\langle \sigma ,\sigma
^{\sharp-1}\right\rangle|_{\phi^{\sharp}(k(Y))}=id_{\phi^{\sharp}(k(Y))}.
\end{equation*}

Hence,
\begin{equation*}
t(\sigma)=\left\langle \sigma ,\sigma ^{\sharp-1}\right\rangle \in Gal\left(k\left(
X\right) /k\left( Y\right) \right) .
\end{equation*}
This also proves that $t$ is well-defined as a map between sets.

On the other hand,
 take any two automorphisms
\begin{equation*}
\sigma =\left( \sigma ,\sigma ^{\sharp }\right) ,\delta =\left( \delta
,\delta ^{\sharp }\right) \in Aut\left( X/Y\right) .
\end{equation*}
We have
\begin{equation*}
\delta ^{\sharp -1}\circ \sigma ^{\sharp -1}=(\delta \circ \sigma )^{\sharp
-1}
\end{equation*} according to preliminary facts on morphisms of schemes (see \cite{ega,hartshorne,iitaka}).

Then
\begin{equation*}
\left\langle \delta ,\delta ^{\sharp -1}\right\rangle \circ \left\langle
\sigma ,\sigma ^{\sharp -1}\right\rangle =\left\langle \delta \circ \sigma
,\delta ^{\sharp -1}\circ \sigma ^{\sharp -1}\right\rangle .
\end{equation*}

Hence,
\begin{equation*}
t:Aut\left( X/Y\right) \rightarrow Gal\left(k\left(
X\right) /k\left( Y\right) \right)
\end{equation*}
is a homomorphism of groups.

\textbf{Step 3.} Prove that
\begin{equation*}
t:Aut\left( X/Y\right) \longrightarrow Gal\left( k\left( X\right) /k\left(
Y\right) \right)
\end{equation*}
is an injective homomorphism.

In deed, take any two automorphisms $\sigma ,\sigma ^{\prime }\in {Aut}\left( X/Y\right) $ such
that $$t\left( \sigma \right) =t\left( \sigma ^{\prime }\right) .$$ We have
\begin{equation*}
\left( \sigma \left( U\right) ,\sigma ^{\sharp-1}\left( f\right) \right)
=\left( \sigma ^{\prime }\left( U\right) ,\sigma ^{\prime \sharp-1}\left(
f\right) \right)
\end{equation*}
for each element $\left( U,f\right) \in k\left( X\right) .$ In particular,
\begin{equation*}
\left( \sigma \left( U_{0}\right) ,\sigma ^{\sharp-1}\left( f\right) \right)
=\left( \sigma ^{\prime }\left( U_{0}\right) ,\sigma ^{\prime
\sharp-1}\left( f\right) \right)
\end{equation*}
holds for any $f\in \mathcal{O}_{X}(U_{0})$ and any affine open subset $U_{0}$ of $X$.

It is seen that
there are three rings
$$
A_0=\mathcal{O}_{X}(U_{0});$$ $$B_0=\mathcal{O}_{X}(\sigma(U_{0}));$$ $$
{ B_{0}}^{\prime}=\mathcal{O}_{X}(\sigma^{\prime}(U_{0})),
$$
as subrings of $k(X)$, satisfying the property
\begin{equation*}
B_0=\sigma^{\sharp-1}(A_0) =\sigma^{\prime \sharp-1}(A_0)=B_{0}^{\prime}.
\end{equation*}

Then
\begin{equation*}
\sigma |_{U_{0}}=\sigma ^{\prime }|_{U_{0}}
\end{equation*}
holds as isomorphisms of schemes in virtue of preliminary facts
on morphisms of affine schemes (see \cite{ega,hartshorne,iitaka}).

Let $U_{0}$ run through all affine open
sets of $X$. We must have
\begin{equation*}
\sigma =\sigma ^{\prime }
\end{equation*}
on the whole of $X$. This proves that $t$ is an injection.

\textbf{Step 4.} Prove that
\begin{equation*}
t:Aut\left( X/Y\right) \longrightarrow Gal\left( k\left( X\right) /k\left(
Y\right) \right)
\end{equation*}
is a surjective homomorphism.

In fact, fixed an automorphism $\rho\in Gal\left(k\left(
X\right) /k\left( Y\right) \right)$.
As
\begin{equation*}
k(X)=\{(U_{f},f):f\in \mathcal{O}_{X}(U_{f})\text{ and }U_{f}\subseteq X
\text{ is open}\},
\end{equation*}
we have
\begin{equation*}
\rho :\left( U_{f},f\right) \in k\left( X\right) \longmapsto \left( U_{\rho
\left( f\right) },\rho \left( f\right) \right) \in k\left( X\right) 
\end{equation*}
according to \emph{Proposition 1.44} of \cite{iitaka},
where $U_{f}$ and $U_{\rho (f)}$ are open sets in $X$, $f$ is
contained in $\mathcal{O}_{X}(U_{f})$, and $\rho (f)$ is contained
in $ \mathcal{O}_{X}(U_{\rho (f)})$.

We will proceed in several sub-steps to prove that for each element
\begin{equation*}
\rho\in Gal\left(k\left(
X\right) /k\left( Y\right) \right)
\end{equation*}
there is a unique element
\begin{equation*}
\lambda\in {Aut}(X/Y)
\end{equation*}
such that
\begin{equation*}
t(\lambda)=\rho.
\end{equation*}

\textbf{Substep 4-a.} Fixed any affine open set $V$ of $Y$. Show that for
each affine open set $U\subseteq \phi^{-1}(V)$ there is an affine open set $%
U_{\rho}$ in $X$ such that $\rho$ determines an isomorphism between affine
schemes $(U,\mathcal{O}_{X}|_{U})$ and $(U_{\rho},\mathcal{O}
_{X}|_{U_{\rho}})$.

In fact, take any affine open sets $V\subseteq Y$ and $U\subseteq
\phi^{-1}(V)$. We have
\begin{equation*}
A= \mathcal{O}_{X}(U) =\{\left( U_{f},f\right) \in k\left( X\right)
:U_{f}\supseteq U\}.
\end{equation*}
Put
\begin{equation*}
B=\{\left( U_{\rho \left( f\right) },\rho \left( f\right) \right) \in
k\left( X\right) :\left( U_{f},f\right) \in A \}.
\end{equation*}

Then there exists an affine open set $W\subseteq \phi^{-1}(V)$ and an isomorphism $$\widetilde{\rho}=\rho\mid_{A}:A\to B=\mathcal{O}_{X}(W)$$  such that $\rho$ induces naturally from $\widetilde{\rho}$ according to the assumption that $X$ is locally complete over $Y$. Write $U_{\rho}\triangleq W$.

Therefore, by $\rho$ we have a unique isomorphism
\begin{equation*}
\lambda_{U}=\left(\lambda_{U}, \lambda_{U}^{\sharp} \right): (U, \mathcal{O}
_{X}|_{U}) \rightarrow (U_{\rho}, \mathcal{O}_{X}|_{U_{\rho}})
\end{equation*}
of the affine open subscheme in $X$ such that
\begin{equation*}
\rho |_{\mathcal{O}_{X}(U)}=\lambda_{U}^{\sharp -1}: \mathcal{O}_{X}(U)
\rightarrow \mathcal{O}_{X}(U_{\rho}).
\end{equation*}

\textbf{Substep 4-b.} Take any affine open sets $V\subseteq Y$ and $%
U,U^{\prime}\subseteq\phi^{-1}(V)$. Show that
\begin{equation*}
\lambda_{U}|_{U\cap U^{\prime}}=\lambda_{U^{\prime}}|_{U\cap U^{\prime}}
\end{equation*}
holds as morphisms of schemes.

In fact, by the above construction in $\emph{Substep 4-a}$, for each $%
\lambda_{U}$ it is seen that $\lambda_{U}^{\sharp}$ and $\lambda^{
\sharp}_{U^{\prime}}$ coincide restricted to the intersection $U\cap
U^{\prime}$ as homomorphisms of rings since we have
\begin{equation*}
\rho |_{\mathcal{O}_{X}(U\cap U^{\prime})}=\lambda_{U}|_{U\cap
U^{\prime}}^{\sharp -1}: \mathcal{O}_{X}(U\cap U^{\prime}) \rightarrow
\mathcal{O}_{X}({(U\cap U^{\prime})}_{\rho});
\end{equation*}
\begin{equation*}
\rho |_{\mathcal{O}_{X}(U\cap U^{\prime})}=\lambda_{U^{\prime}}|_{U\cap
U^{\prime}}^{\sharp -1}: \mathcal{O}_{X}(U\cap U^{\prime}) \rightarrow
\mathcal{O}_{X}({(U\cap U^{\prime})}_{\rho}).
\end{equation*}

On the other hand, for any point $x\in U\cap U^{\prime}$, we must have
\begin{equation*}
\lambda_{U}(x)=\lambda_{U^{\prime}}(x).
\end{equation*}

Otherwise, if $\lambda _{U}(x)\not=\lambda _{U^{\prime }}(x)$, will have an
affine open subset $X_{0}$ of $X$ that contains either of the points $%
\lambda _{U}(x)$ and $\lambda _{U^{\prime }}(x)$ but does not contain the
other since the underlying space of $X$ is a Kolmogorov space
(see \cite{ega,hartshorne,iitaka}). Assume $\lambda _{U}(x)\in X_{0}$ and $\lambda
_{U^{\prime }}(x)\not\in X_{0}$. We choose an affine open subset $U_{0}$ of $%
X$ such that $x\in U_{0}\subseteq U\cap U^{\prime }$ and $\lambda
_{U}(U_{0})\subseteq X_{0}$ since we have
\begin{equation*}
\lambda _{U}(U\cap U^{\prime })={(U\cap U^{\prime })}_{\rho }\subseteq
U_{\rho };
\end{equation*}
\begin{equation*}
\lambda _{U^{\prime }}(U\cap U^{\prime })={(U\cap U^{\prime })}_{\rho
}\subseteq U_{\rho }^{\prime }.
\end{equation*}
However, for each $\lambda _{U}$, we have
\begin{equation*}
\lambda _{U}(U_{0})=(U_{0})_{\rho }=\lambda _{U^{\prime }}(U_{0});
\end{equation*}
then
\begin{equation*}
\lambda _{U^{\prime }}(x)\in (U_{0})_{\rho }\subseteq X_{0},
\end{equation*}
where there will be a contradiction.

Hence, $\lambda _{U}$ and $\lambda
_{U^{\prime }}$ coincide on $U\cap U^{\prime }$ as mappings of topological
spaces.

\textbf{Substep 4-c.} By gluing the $\lambda_{U}$ along all such affine open
subsets $U$, we have a homeomorphism $\lambda$ of $X$ onto $X$ as a
topological space given by
\begin{equation*}
\lambda: x\in X \mapsto \lambda_{U}(x)\in X
\end{equation*}
where $x$ belongs to $U$ and $U$ is an affine open subset of $X$. We have $$
\lambda|_{U}=\lambda_{U}. $$

By $\emph{Substep4-b}$ it is seen that $\lambda $ is well-defined. Then we
have an automorphism, namely $\lambda $, of the scheme $(X,\mathcal{O}_{X})$
(see \cite{ega,hartshorne,iitaka}).

\textbf{Substep 4-d.} Show that $\lambda$ is contained in $Aut\left(
X/Y\right)$ satisfying
\begin{equation*}
t\left(\lambda\right)=\rho.
\end{equation*}

In deed, as $\rho$ is an isomorphism of $k(X)$ over $k(Y)
$, the isomorphism $\lambda_{U}$ is over $Y$ by $\phi$ for any affine
open subset $U$ of $X$; then $\lambda$ is an automorphism of $X$ over $Y$ by
$\phi$ such that $$t\left(\lambda\right)=\rho.$$

This proves that for each $\rho \in Gal\left(k\left(
X\right) /k\left( Y\right) \right) $, there exists $%
\lambda\in Aut\left( X/Y\right) $ such that
\begin{equation*}
t(\lambda)=\rho.
\end{equation*}
Hence, ${t}$ is surjective.

Therefore, by \emph{Steps 1-4} above, it is seen that there is a
group isomorphism
$$
t:Aut\left( X/Y\right) \longrightarrow Gal\left( k\left( X\right) /k\left(
Y\right) \right)
$$
by
\begin{equation*}
\sigma =(\sigma ,\sigma ^{\sharp })\longmapsto t(\sigma )=\left\langle
\sigma ,\sigma ^{\sharp -1}\right\rangle,
\end{equation*}
where $\left\langle \sigma ,\sigma ^{\sharp -1}\right\rangle $ is the map of
$k(X)$ into $k(X)$ given by
\begin{equation*}
\left( U,f\right) \in \mathcal{O}_{X}(U)\subseteq k\left( X\right)
\longmapsto \left( \sigma \left( U\right) ,\sigma ^{\sharp -1}\left(
f\right) \right) \in \mathcal{O}_{X}(\sigma (U))\subseteq k\left( X\right)
\end{equation*}
for any open set $U$ in $X$ and any element $f\in \mathcal{O}_{X}(U)$.

This completes the proof.
\end{proof}

\section{The Calculus and Construction of Galois Covers}

In this section
we will give a universal construction of Galois covers over a given integral variety.
For an arithmetic variety $X$, in general, we will have  many distinct Galois
covers $X_{et}$ over $X$ for the \'{e}tale fundamental group
$\pi^{et}_{1}(X;s)$ such that each $X_{et}$ is a
realization for the group $\pi^{et}_{1}(X;s)$, i.e., there will be isomorphisms
$$k(X_{et})\cong k(X;\Omega)^{un};$$
$$
Gal(k(X;\Omega)^{un}/k(X))\cong Aut(X_{et}/X)
$$
both in a natural manner.

Furthermore, all these Galois covers $X_{et}$ of $X$ can
have the same function fields $k(X;\Omega)^{un}$ but can not be isomorphic as schemes.
In other words, for any such Galois covers $X_{et}$ and $X_{et}^{\prime}$ over $X$, we have
$$X_{et}\not\cong X_{et}^{\prime};$$
$$k(X_{et})\cong k(X_{et}^{\prime});$$
$$Aut(X_{et}/X)\cong Aut(X_{et}^{\prime}/X).$$

\subsection{Calculus and existence of Galois covers}

Here we give the existence of  Galois covers of an integral variety such that the
Galois covers can be prescribed by any given
Galois extensions of the function field.

\begin{theorem}
For a field $L$ and an integral variety $Y$ such that $L$ is an algebraic Galois extension of $
k(Y)$, there exists an integral variety $X_{L}$ and an affine surjective morphism $f_{L}:X_{L}\rightarrow Y$ satisfying the properties:
\begin{itemize}
\item $L=k\left( X_{L}\right) $;

\item $X_{L}$ is essentially affine in $L$;

\item $X_{L}$ is Galois over $Y$ by $f_{L}$;

\item $X_{L}$ is locally quasi-galois over $Y$ by $f_{L}$;

\item $Aut(X_{L}/Y)$ acts on the fiber $f_{L}^{-1}(y)$ transitively for any $
y\in Y$;

\item For any affine open set $V$ of $Y$, there is an isomorphism
$$\mathcal{O}_{Y}(V)\cong{\mathcal{O}_{X_{L}}}
^{Aut(X_{L}/Y)}(f_{L}^{-1}(V))$$ of groups in a natural manner.
\end{itemize}
\end{theorem}

The above integral variety $X_{L}$ with a
morphism $f_{L}$, denoted by $(X_{L},f_{L})$, will be called a \textbf{\
Galois cover} of $Y$ \textbf{in a field} $L$.

Now consider the calculus of Galois covers of an integral variety.

\begin{definition}
Let $Y$ be an integral variety and let $L$ be an algebraic Galois extension of $k(Y)$.
 Put
\begin{itemize}
\item $\Sigma[L/Y]\triangleq$ the collection of all the sets $\Delta$ of generators of $L$ over $k(Y)$ such that $\Delta\subseteq L \setminus k(Y)$;

\item $Y[\Delta]\triangleq$ the integral variety $X$, a Galois cover of $Y$ in $L$, obtained by adding a set $\Delta=\Delta(L/k(Y))\in \Sigma[L/Y]$ to an essentially affine realization of $Y$ in the
universal construction of Galois covers in \S 4.2 below.
\end{itemize}
\end{definition}

It is clear that $Y[\Delta]$ has all the properties in \emph{Theorem 4.1} above.

\begin{remark}
Fixed an integral variety $Y$ and a Galois extension $L$ of $k(Y)$.
There are the following statements.
\begin{itemize}
\item Different sets $\Delta\in \Sigma[L/Y]$  can produce different Galois covers $Y[\Delta]$ of $Y$ in $L$. In general, it is not true that $Y[\Delta]$ and $Y[\Delta^{\prime}]$ are isomorphic over $Y$ for any $\Delta,\Delta^{\prime}\in \Sigma[L/Y]$.

\item There is a biggest Galois cover $Y[\Delta_{min}]$ of $Y$ in $L$, called \textbf{sp-complete}, which contains all possible points in the underlying space. For instance, take a minimal element $\Delta_{min}$ of $\Sigma[L/Y]$.

\item There is a smallest Galois cover $Y[\Delta_{max}]$ of $Y$ in $L$, which has the same underlying space as $Y$. For instance, take a maximal element $\Delta_{max}$ of $\Sigma[L/Y]$.
\end{itemize}
\end{remark}

\begin{proposition}
Let $Z$ be an integral variety and $L$ an algebraic Galois extension of $k(Z)$.
Then for any $\Delta,\Delta^{\prime}\in \Sigma[L/Z]$, there is an isomorphism
\begin{equation*}
Z[\Delta,{\Delta}^{-1}]\cong Z[\Delta^{\prime},{\Delta^{\prime}}^{-1}]
\end{equation*}
between schemes over $Z$. In particular, each $Z[\Delta,{\Delta}^{-1}]$ is a smallest Galois cover of $Z$.
 Here, $Z[\Delta,{\Delta}^{-1}]$ denotes $Z[\Delta \cup {\Delta}^{-1}]$.
\end{proposition}

\begin{proof}
It reduces to consider the affine
scheme $Z=Spec(A)$ according to the construction in \S 4.2 below. But this is trivial.
\end{proof}

\subsection{A universal construction of Galois covers}

In the following we will prove \emph{Theorem 4.1}.

\begin{proof}
(\textbf{Proof of \emph{Theorem 4.1}})
Here we take the trick in \cite{an3}.

From \emph{Lemma 2.3}, assume that $Y$ is essentially affine in the function field $k(Y)$ without loss
of generality.

Denote by $\mathcal{C}_{Y}$ the set
of local charts $\left( V,\psi _{V},B_{V}\right)$ such that $B_{V}=\mathcal{%
O }_{Y}\left( V\right)$, where $V$ runs through all affine open sets of $Y$.
Let
\begin{equation*}
\Delta(L/k(Y))\subseteq L\setminus k(Y)
\end{equation*}
be a set of generators of the field $L$ over the function field $k(Y)$.

Following the procedure developed in \cite{an3}, we will proceed in several
steps to construct an integral variety $X=X_{L}$ and a morphism $f=f_{L}:X_{L}\to Y$.

{\textbf{Step 1.}} For any local chart $\left( V,\psi _{V},B_{V}\right)
\in \mathcal{C}_{Y}$, define
\begin{equation*}
A_{V}\triangleq B_{V}\left[ \Delta _{V}\right],
\end{equation*}
i.e., the subring of $L$ generated over $B_{V}$ by the set
\begin{equation*}
\Delta _{V}\triangleq \{\sigma \left(w\right) \in L:\sigma \in Gal\left(
L/k(Y)\right) ,w\in \Delta(L/k(Y))\}.
\end{equation*}

Then $Fr\left( A_{V}\right) =L $ holds. It is seen that $B_{V}$ is exactly
the invariant subring of the natural action of the Galois group $Gal\left(
L/M\right) $ on $A_{V}.$

Set
\begin{equation*}
i_{V}:B_{V}\rightarrow A_{V}
\end{equation*}
to be the inclusion.

{\textbf{Step 2.}} Define the disjoint union
\begin{equation*}
\Sigma =\coprod\limits_{\left( V,\psi _{V},B_{V}\right) \in \mathcal{C}
_{Y}}Spec\left( A_{V}\right) .
\end{equation*}

Then $\Sigma $ is a topological space, where the topology $\tau _{\Sigma }$
on $\Sigma $ is naturally determined by the Zariski topologies on all $%
Spec\left( A_{V}\right) .$

Let
\begin{equation*}
\pi _{Y}:\Sigma \rightarrow Y
\end{equation*}
be the projection.

{\textbf{Step 3.}} Given an equivalence relation $R_{\Sigma}$ in $%
\Sigma $ in such a manner:

For any $x_{1},x_{2}\in \Sigma $, we say
\begin{equation*}
x_{1}\sim x_{2}
\end{equation*}
if and only if the two conditions are satisfied:

\begin{itemize}
\item $\pi _{Y}(x_{1})=\pi _{Y}(x_{2})$ holds as points in $Y$;

\item $x_{1}$ and $x_{2}$ are $L$-equivalent as affine $L$-points.
\end{itemize}

Let
\begin{equation*}
X=\Sigma /\sim
\end{equation*}
and let
\begin{equation*}
\pi _{X}:\Sigma \rightarrow X
\end{equation*}
be the projection.

It is seen that $X$ is a topological space as a quotient of $\Sigma .$

{\textbf{Step 4.}} Define a map
\begin{equation*}
f:X\rightarrow Y
\end{equation*}
by
\begin{equation*}
\pi _{X}\left( z\right) \longmapsto \pi _{Y}\left( z\right)
\end{equation*}
for each $z\in \Sigma $.

{\textbf{Step 5.}} Define
\begin{equation*}
\mathcal{C} _{X}=\{\left( U_{V},\varphi _{V},A_{V}\right) \}_{\left( V,\psi
_{V},B_{V}\right) \in \mathcal{C}_{Y}}
\end{equation*}
where $U_{V}\triangleq\pi _{Y}^{-1}\left( V\right) $ is an open set in $X$
and $\varphi _{V}:U_{V}\rightarrow Spec(A_{V})$ is the identity map on $%
U_{V} $ for each $\left( V,\psi _{V},B_{V}\right) \in \mathcal{C}_{Y}$.

We have an integral scheme $\left( X,\mathcal{O}_{X}\right) $ by gluing the
affine schemes $Spec\left( A_{V}\right) $ for all local charts $\left(
V,\psi _{V},B_{V}\right) \in \mathcal{C}_{Y}$ with respect to the
equivalence relation $R_{\Sigma }$ (see \cite{ega,hartshorne}).

It is seen that $\mathcal{C} _{X}$ is a reduced affine covering of $X$ with
values in $L$. In particular, $\mathcal{C} _{X}$ is quasi-galois closed over
$Y$ by $f$.

By \emph{Lemma 3.5} it is seen that $X$ and $f$ have the desired
properties.

This completes the proof.
\end{proof}

\subsection{A realization of the \'{e}tale fundamental group}

From \emph{Theorem 4.1} we have the following realization for the \'{e}tale fundamental
group of an arithmetic scheme.

\begin{lemma}
Let $X$ be an arithmetic variety  and $\Omega$ an algebraic closure  of $k(X)$. Then there exists an
integral variety $X_{et}$ and an affine surjective morphism
$$
f_{et}:X_{et}\rightarrow X
$$
such that
\begin{itemize}
\item $k\left( X_{et}\right) =k(X;\Omega)^{un}$;

\item $(X_{et},f_{et})$ is a Galois cover of $X$ in the sense of Theorem 4.1.
\end{itemize}
\end{lemma}

Such a Galois cover $(X_{et},f_{et})$ will be called
a \textbf{universal Galois cover} of $X$ for
the \'{e}tale fundamental group $\pi _{1}^{et}\left( X,s\right) $.
Here $s$ is a geometric point of $X$ over $\Omega$.

\begin{remark}
Let $X$ be an arithmetic variety and $\Omega$ an algebraic closure of $k(X)$.
Then for any $$\Delta_{un}\in \Sigma[k(X;\Omega)^{un}/X],$$ $X[\Delta_{un}]$ is a
universal Galois cover of $X$ for the \'{e}tale fundamental group
$\pi _{1}^{et}\left( X,s\right) $. In general,  there can be many distinct universal
Galois covers of $X$ for $\pi _{1}^{et}\left( X,s\right) $ according to the calculus of Galois covers in \S 4.1.
\end{remark}

\section{Formally Unramified = Arithmetically Unramified}

In this section we will prove that for an arithmetic scheme, the maximal formally unramified extension is equal to the maximal arithmetically
unramified extension.

At the same time, arithmetically
unramified, which is defined in an evident manner, coincides with that in algebraic number theory and has many arithmetic operations and properties.
Hence, via arithmetically unramified, it is seen the maximal formally unramified extension is exactly of an arithmetic
sense.

On the other hand, it is seen that arithmetically unramified is more accessible to operations (see \cite{an6,an7}); it follows
that arithmetically unramified extensions
can serve as a computation of the \'{e}tale fundamental group of
an arithmetic variety.

\subsection{Arithmetic function fields and arithmetically  unramified extensions}

Recall that an \textbf{arithmetic function field} \textbf{over variables}
$t_{1},t_{2},\cdots, t_{n}$ is the fractional field of an integral
domain
\begin{equation*}
\mathbb{Z}[t_{1},t_{2},\cdots, t_{n},s_{1},s_{2},\cdots, s_{m}]
\end{equation*}
of finite type over $\mathbb{Z}$, where
\begin{equation*}
s_{1},s_{2},\cdots, s_{m}
\end{equation*}
are algebraic over $\mathbb{Q}[t_{1},t_{2},\cdots, t_{n}]$ and
\begin{equation*}
t_{1},t_{2},\cdots, t_{n},
\end{equation*}
are a finite number of (algebraic independent) variables over $\mathbb{Q}$.

\begin{definition} (see \cite{an4,an6,an7})
Let $K$ be an arithmetic function field over variables $t_{1},t_{2},\cdots,
t_{n}$. As in algebraic number theory, put
\begin{itemize}
\item $\mathcal{O}_{[t_{1},t_{2},\cdots, t_{n}]K}\triangleq$ the subring of
elements $x\in K$ that are integral over the integral domain $\mathbb{Z}
[t_{1},t_{2},\cdots, t_{n}]$.
\end{itemize}

The ring $\mathcal{O}_{[t_{1},t_{2},\cdots, t_{n}]K}$ is said to be the
\textbf{ring of algebraic integers} of $K$ \textbf{over variables} $
t_{1},t_{2},\cdots, t_{n}$.
\end{definition}

It is easily seen that $K=Fr(\mathcal{O}_{K})$ holds. Sometimes, we will write $$\mathcal{O}_{K}=%
\mathcal{O}_{[t_{1},t_{2},\cdots, t_{n}]K}$$ for brevity.

Let $K$ and $L$ be two arithmetic function fields  over variables $t_{1},t_{2},\cdots, t_{n}$. We say
\begin{equation*}
(L,\mathcal{O}_{L})\supseteq (K,\mathcal{O}_{K})
\end{equation*}
if and only if there are the properties
\begin{equation*}
\overline{K}\supseteq L\supseteq K;
\end{equation*}
\begin{equation*}
\mathcal{O}_{L}\supseteq \mathcal{O}_{K}\supseteq \mathbb{Z}
[t_{1},t_{2},\cdots, t_{n}].
\end{equation*}

\begin{definition} (see \cite{an4,an6,an7})
Let $L\supseteq K$ be any two arithmetic function fields over variables $t_{1},t_{2},\cdots,
t_{n}$. The field $L$ is  \textbf{arithmetically unramified} over $K$ (\textbf{%
relative to variables} $t_{1},t_{2},\cdots, t_{n}$) if for the pairs $(L,%
\mathcal{O}_{L})\supseteq (K, \mathcal{O}_{K})$, the affine scheme $Spec(%
\mathcal{O}_{L})$ is unramified over $Spec(\mathcal{O}_{K})$ by the morphism
induced from the inclusion map $K\hookrightarrow L $.
\end{definition}

It is clear that the arithmetically unramified extensions coincide exactly with those in
algebraic number theory (for instance, see \cite{neu}).

\begin{lemma}
\emph{(see \cite{an4,an6,an7})}
The arithmetically unramified extension is
independent of the choice of variables $t_{1},t_{2},\cdots, t_{n}$.
\end{lemma}

\begin{proof}
It is immediate from operations on arithmetically unramified extensions in \S 5.2 below.
\end{proof}

\subsection{Operations on arithmetically unramified extensions}

As in algebraic number theory, we have
several operations on arithmetically unramified
extensions such as base changes, composites, subfields, and transitivity.

Let $K\subseteq L\subseteq M$ be arithmetic function fields over variables $%
t_{1},t_{2},\cdots, t_{n}$. We have the following lemmas from preliminary facts on unramified morphisms and
\'{e}tale morphisms between schemes (see \cite{f-k,sga1,mln,sz}).

\begin{lemma}
\emph{(see \cite{an4,an6,an7})}
Let $L/K$ and $M/L$ be arithmetically unramified. Then so is $%
M/L$.
\end{lemma}

\begin{lemma}
\emph{(see \cite{an4,an6,an7})}
Let $L$ be arithmetically unramified over $K$. Take a subfield $L\subseteq
L_{0}\subseteq K$. Then $L/L_{0}$ and $L_{0}/K$ are arithmetically
unramified.
\end{lemma}

\begin{lemma}
\emph{(see \cite{an4,an6,an7})}
Let $L_{1}$ and $L_{2}$ be two arithmetic function fields over variables $%
t_{1},t_{2},\cdots, t_{n}$. Suppose that $L_{1}$ and $L_{2}$ are
arithmetically unramified over $K$, respectively. Then the composite $%
L_{1}\cdot L_{2}$ of fields is also arithmetically unramified over $K$.
\end{lemma}

\begin{lemma}
\emph{(see \cite{an4,an6,an7})}
Let $K^{\prime}\supseteq K$ be an arithmetic function field over variables $%
t_{1},t_{2},\cdots, t_{n}$. Suppose that $L$ is arithmetically unramified
over $K$. Then the composite $L^{\prime}=L\cdot K^{\prime}$ of fields is
also arithmetically unramified over $K^{\prime}$.
\end{lemma}

\subsection{The maximal formally unramified extension is nothing other than the maximal arithmetically unramified extension}

Now take the maximal arithmetically unramified extension of an arithmetic
variety as one does in algebraic number theory, which will be proved to be equal to the maximal formally
unramified extension of the arithmetic variety.

Let $L$ be an arithmetic function field over variables $t_{1},t_{2},\cdots,
t_{n}$. Let $\Omega$ be an algebraic closure of $L$.

Put
\begin{itemize}
\item $\left[ L \right](\Omega)^{au}\triangleq$ the set of all finite
arithmetically unramified subextensions of $L$ contained in $\Omega$.

\item $(L;\Omega)^{au}\triangleq {{\lim}_{\rightarrow_{Z\in {\left[ L \right]%
(\Omega)^{au}}}}}{\ \lambda_{Z}(Z)}$, i.e., the direct limit of the direct
system of rings indexed by $\left[ L \right](\Omega)^{au}$, where each $%
\lambda_{Z}:Z\in {\left[ L \right](\Omega)^{au}}\to \Omega$ is the $L$%
-embedding,
\end{itemize}
where $\left[ L;\Omega \right]^{au}$ is taken as a directed set defined by set
inclusion.

The subfield $(L;\Omega)^{au}$ of $\Omega$ is called the \textbf{maximal
arithmetically unramified extension} of $L$ in $\Omega$.

Write
\begin{itemize}
\item $k(X;\Omega)^{au}\triangleq (k(X);\Omega)^{au}$
for an arithmetic variety $X$ and an algebraic closure $\Omega$ of the
function field $k(X)$.
\end{itemize}

\begin{lemma}
\emph{(see \cite{an4,an6,an7})}
Let $\left[ L \right](\Omega)^{au}_{0}$ be the set of all finite
arithmetically unramified Galois subextensions of $L$ contained in $\Omega$.
Then we have
\begin{equation*}
(L;\Omega)^{au}= {{\lim}_{\rightarrow_{Z\in {\left[ L \right]%
(\Omega)^{au}_{0}}}}}{\ \lambda_{Z}(Z)}.
\end{equation*}
Moreover, $(L;\Omega)^{au}$ is an algebraic Galois extension of $L$.
\end{lemma}

\begin{proof}
It is immediate from the preliminary fact that $\left[ L \right]%
(\Omega)^{au}_{0}$ is a cofinal directed subset of $\left[ L \right]%
(\Omega)^{au}$.
\end{proof}

Fortunately there exists the following lemma which says that for an arithmetic variety,  the maximal arithmetically
unramfied extension is equal to the maximal formally ramified
extension.

\begin{lemma}
Let $X$ be an arithmetic variety and let $\Omega$ be an algebraic closure of the function field $k(X)$. Then we have
\begin{equation*}
k(X;\Omega)^{au}= k(X;\Omega)^{un}
\end{equation*}
as subfields of $\Omega$.
\end{lemma}

\begin{proof}
Let $L=k(X)$. As $L$ is an arithmetic function field over
several variables, consider the ring $\mathcal{O}_{L}$ of algebraic integers of $L
$. We have
\begin{equation*}
L=k(Spec(\mathcal{O}_{L});
\end{equation*}
then
\begin{equation*}
(L;\Omega)^{un}=(k(Spec(\mathcal{O}_{L}));\Omega)^{un}.
\end{equation*}

To prove the lemma, it suffices to prove
\begin{equation*}
(k(Spec(\mathcal{O}_{L}));\Omega)^{un}=(L;\Omega)^{au}
\end{equation*}
holds.

In deed, we have
\begin{equation*}
(L;\Omega)^{au}= {{\lim}_{\rightarrow_{Z\in {\left[ L \right]%
(\Omega)^{au}_{0}}}}}{\ \lambda_{Z}(Z)};
\end{equation*}
\begin{equation*}
(k(Spec(\mathcal{O}_{L}));\Omega)^{un}= {{\lim}_{\rightarrow_{X\in {\left[%
Spec(\mathcal{O}_{L}); \Omega \right]_{et}}}}}{\lambda_{X}(k(X))}.
\end{equation*}

On the other hand, put
\begin{equation*}
\Sigma=\{k(X)\mid X\in {\left[Spec(\mathcal{O}_{L});
\Omega \right]_{et}}\}.
\end{equation*}
Then $\Sigma$ is a directed set, where the partial order is given by set inclusion.

As $Spec(\mathcal{O}_{L})$ is a normal scheme, it is easily seen that each $X\in {%
\left[Spec(\mathcal{O}_{L}); \Omega \right]_{et}}$ is normal from
preliminary facts on \'{e}tale covers (see \cite{sga1,mln,sz}). It follows that $\left[ L %
\right](\Omega)^{au}_{0}$ is a cofinal directed subset of $\Sigma$.

Hence,
\begin{equation*}
\begin{array}{l}
(k(Spec(\mathcal{O}_{L}));\Omega)^{un} \\
={{\lim}_{\rightarrow_{X\in {\left[Spec(\mathcal{O}_{L}); \Omega \right]_{et}%
}}}}{\lambda_{X}(k(X))} \\
= {{\lim}_{\rightarrow_{Z\in {\Sigma}}}}{\lambda_{Z}(Z)}
\\
={{\lim}_{\rightarrow_{Z\in {\left[ L \right](\Omega)^{au}_{0}}}}}{\
\lambda_{Z}(Z)} \\
=(L;\Omega)^{au}.%
\end{array}%
\end{equation*}
This completes the proof.
\end{proof}

By \emph{Lemma 5.9} we have such a computation of the \'{e}tale fundamental group of an arithmetic
variety.

\begin{remark}
For any arithmetic variety $X$, there is an isomorphism
\begin{equation*}
\pi _{1}^{et}\left( X,s\right) \cong Gal\left( {k(X; \Omega})^{au}/k\left( X\right)
\right)
\end{equation*}
 of groups in a natural manner. Here, $s$ is a geometric point of $X$ over an algebraic closure $\Omega$ of the function field $k(X)$.
\end{remark}

\section{Proof of the Main Theorem}

\subsection{Recalling the \'{e}tale fundamental group}

Fixed a connected scheme $Z$. By a \textbf{finite \'{e}tale cover} of $Z$
we understand a scheme that is finite and \'{e}tale over $Z$
by a surjective morphism. Denote by $$Fet\left[ Z\right] $$ the set of all finite \'{e}tale covers of $Z$.

Fixed a geometric point $s$ of $Z$ over a given separable closure $\Omega $
of $k\left( Z\right) $. There is a fiber functor $Fib_{s}$ on $Fet\left[ Z %
\right] $ given by $$Fib_{s}\left( Y\right) =sp\left( Y\times _{Z}Spec\left(
\Omega \right) \right) $$ for any $Y\in Fet\left[ Z\right].$ Here $sp\left(
W\right) $ denotes the underlying space of a scheme $W.$

Then the \textbf{\'{e}tale fundamental group} of $Z$ over $s,$ denoted by $%
\pi _{1}^{et}\left( Z;s\right) ,$ is defined to be the group of
automorphisms of the fiber functor $Fib_{s}$ over $Z$. See \cite{sga1,mln,sz} for details.

Let $Y$ be a finite \'{e}tale cover of $Z$ by a morphism $f$ and let $t$ be
a geometric point of $Y$ over $s$ with $s=f\circ t.$ Then $\left( Y,t\right)
$ is said to be a \textbf{pointed finite \'{e}tale cover} of $\left(
Z,s\right) .$

\begin{remark}
Let $Fet\left[ Z;\Omega \right] $ denote the set of pointed finite \'{e}tale
covers of $\left( Z,s\right) .$ For any $\left( Y_{1},t_{1}\right) ,\left(
Y_{2},t_{2}\right) \in Fet\left[ Z;\Omega \right] ,$ we say $\left(
Y_{1},t_{1}\right) \leq \left( Y_{2},t_{2}\right) $ if and only if $\left(
Y_{2},t_{2}\right) $ is a pointed finite \'{e}tale cover of $\left(
Y_{1},t_{1}\right) .$ Then $Fet\left[ Z;\Omega \right] $ is a directed set
with $\leq .$
\end{remark}

\begin{remark}
There is an isomorphism $$\pi _{1}^{et}\left( Z;s\right) \cong
{\lim_{\leftarrow}}_{ \left( Y,t\right) \in Fet\left[ Z;\Omega \right] }Aut\left(
Y/Z\right) $$ as inverse limits of inverse systems of groups.
\end{remark}

In practice, we usually use finite \'{e}tale Galois covers to compute the inverse limit, i.e., the
\'{e}tale fundamental group. Recall that $\left( Y,t\right) \in Fet\left[
Z;\Omega \right] $ is said to be a \textbf{finite \'{e}tale Galois cover} of
$\left( Z,s\right) $ if $Y$ is connected and $\#Aut\left( Y/Z\right) $ is
equal to $\#Fib_{s}\left( Y\right) .$ It is seen that $\#Aut\left(
Y/Z\right) $ and $\#Fib_{s}\left( Y\right) $ are equal if and only if $Aut\left( Y/Z\right) $
induces a transitive action on $Fib_{s}\left( Y\right) $ (see \cite{sga1,mln,sz}).

Let $\left[ X;\Omega \right] _{et}^{\ast }$ be the set of finite \'{e}tale
Galois covers of $\left( Z,s\right) ,$ where any two $\left(
Y_{1},t_{1}\right) ,\left( Y_{2},t_{2}\right) \in \left[ X;\Omega \right]
_{et}^{\ast }$ are identified with each other if there is an isomorphism of $%
\left( Y_{1},t_{1}\right) $ onto $\left( Y_{2},t_{2}\right) $ over $\left(
Z,s\right) .$

\begin{remark}
As a directed set, $\left[ X;\Omega \right] _{et}^{\ast }$ is isomorphic to
a cofinal subset of $Fet\left[ Z;\Omega \right] .$
\end{remark}

\begin{remark}
There are isomorphisms $$\pi _{1}^{et}\left( Z;s\right) \cong
{\lim_{\leftarrow}}_{ \left( Y,t\right) \in Fet\left[ Z;\Omega \right] }Aut\left(
Y/Z\right) \cong {\lim_{\leftarrow}}_{ \left( Y,t\right) \in \left[ X;\Omega %
\right] _{et}^{\ast }}Aut\left( Y/Z\right) $$ as inverse limits of inverse
systems of groups.
\end{remark}

\subsection{A preliminary fact}

Now let $X$ be an arithmetic variety. It is seen that $X$ is connected since
an irreducible space must be connected. Let $s$ be a geometric point
of $X$ over a given separable closure $\Omega $ of $k\left( X\right) $.
As introduced in \S 1, $\left[ X;\Omega %
\right] _{et}$ consists of elements $\left(
Y,y\right) \in \left[ X;\Omega \right] _{et}^{\ast }$ such that $Y$ is
irreducible. Then we have $$\left[ X;\Omega \right] _{et}\subseteq \left[
X;\Omega \right] _{et}^{\ast }$$ as a subset.

\begin{lemma}
Every finite \'{e}tale Galois cover
of an arithmetic variety is irreducible. That is, for any arithmetic variety $X$,
we have $$\left[ X;\Omega \right] _{et}^{\ast
}=\left[ X;\Omega \right] _{et}.$$
\end{lemma}

\begin{proof}
Let $\left( Y,y\right) \in \left[ X;\Omega \right] _{et}^{\ast }.$ By taking a
section of $Y$ over $X,$ it is seen that the underling spaces $$sp\left(
X\right)\cong  sp\left( Y\right) $$ are homeomorphic since $X$ and $Y$ are
connected. As $X$ is irreducible, $Y$ must be irreducible.
\end{proof}

\subsection{Proof of the main theorem}

Now we can prove \emph{Theorem 1.1}.

\begin{proof}
(\textbf{Proof of \emph{Theorem 1.1}}) Let $\Omega$ be  an algebraic closure
of the function field $k\left( X\right)$. By \emph{Lemma 2.3}, assume that
$X $ is essentially affine in $\Omega $ without loss of generality. For brevity, write
$$k(X)^{un}=k(X;\Omega)^{un}.$$

Let $\Delta \subseteq k(X)^{un}\setminus k\left( X\right) $ be a set of
generators of the field $k(X)^{un}$ over $k(X) $. Put
$$
I=\{\text{finite subsets of }\Delta \};$$ $$ G=Gal\left( k(X)^{un}/k\left(
X\right) \right) .
$$

We will proceed in several steps in the following.

\textbf{Step 1.} As in \S 4.1, for $\Delta $ and $G$ we have an
integral variety $$X_{\infty }^{un}=X[\Delta]$$ and a surjective morphism $$
f_{\infty }:X_{\infty }^{un}\rightarrow X$$ such that
\begin{itemize}
\item $k\left( X_{\infty }^{un}\right) =k\left( X\right) ^{un}$;

\item $X_{\infty }^{un}$ is Galois over $X$ by $f_{\infty }$;

\item $X_{\infty }^{un}$ is essentially affine in $k\left( X;s\right)
^{un}\subseteq\Omega$.
\end{itemize}

\textbf{Step 2.} Let ${\alpha }\in I$. Likewise, for $\alpha $ and
$G_{\alpha}=Gal(k(X)[\alpha]/k(X))$ we have an integral variety $$X_{\alpha
}^{un}=X[\alpha]$$ and a surjective morphism $$f_{\alpha }:X_{\alpha
}^{un}\rightarrow X$$ satisfying the properties:
\begin{itemize}
\item $k\left( X_{\alpha }^{un}\right)\subseteq k\left( X\right) ^{un}$;

\item $X_{\alpha }^{un}$ is Galois over $X$ by $f_{\alpha}$;

\item $X_{\alpha}^{un}$ is essentially affine in $k\left( X_{\alpha
}^{un}\right)$.
\end{itemize}

\textbf{Step 3.} Fixed any $\alpha \subseteq \beta$ in $I.$
 We have integral varieties $X_{{\alpha }
}^{un}$ and $X_{{\beta }}^{un}$ and a surjective morphism
\begin{equation*}
f_{\alpha }^{\beta }:X_{\beta }^{un}\rightarrow X_{\alpha }^{un}
\end{equation*}
of finite type, where $f_{\alpha }^{\beta }$ is obtained as in \S 4.2 in a natural manner and satisfies the property
\begin{equation*}
f_{\beta }=f_{\alpha }\circ f_{\alpha }^{\beta }.
\end{equation*}

Then we have Galois covers
$$X_{{\alpha }}^{un}/X;$$ $$X_{{\beta }}^{un}/X;$$ $$X_{{\beta }}^{un}/X_{{\alpha }}^{un}.$$

As $\Delta$ is an infinite set, there is a $\gamma $ in $I$ such that $$\gamma \supseteq \alpha ,\gamma
\supseteq \beta .$$ In particular, $$X_{{\gamma }}^{un}/X_{{\beta }}^{un}$$ and
$$X_{{\gamma }}^{un}/X_{{\alpha }}^{un}$$ are both Galois covers.

\textbf{Step 4.} For any $\alpha ,\beta \in I$, we say $\alpha \leq \beta $
if and only if $\alpha \subseteq \beta .$ Then with $\leq$, $I$ is a partially ordered
set.

We have a direct system $$\{k\left( X_{{\alpha }}^{un}\right) ;i_{\alpha }^{\beta }\}_{\alpha
\in I}$$ of rings, where each
\begin{equation*}
i_{\alpha }^{\beta }:k\left( X_{{\alpha }}^{un}\right) \rightarrow k\left(
X_{{\beta }}^{un}\right)
\end{equation*}
is a homomorphism of fields induced by the morphism $f_{\alpha }^{\beta }$.

For the fields, we have
\begin{equation*}
k\left( X_{\infty }^{un}\right) ={\lim_{\longrightarrow }}_{\alpha \in
I}k\left( X_{{\alpha }}^{un}\right).
\end{equation*}
For the Galois groups, we have
\begin{equation*}
Gal\left( k\left( X_{\infty }^{un}\right) /k\left( X\right) \right) \cong {\
\lim_{\longleftarrow }}_{\alpha \in I}Gal\left( k\left( X_{\alpha
}^{un}\right) /k\left( X\right) \right) .
\end{equation*}

\textbf{Step 5.} Put
\begin{equation*}
\left[X; \Omega \right]_{qc} =\{X_{\alpha }^{un}:\alpha \in I\}.
\end{equation*}

For any $X_{\alpha }^{un},X_{\beta }^{un}\in \left[X; \Omega \right]_{qc} ,$ we say $$
X_{\alpha }^{un}\leq X_{\beta }^{un}$$ if and only if $X_{{\beta }}^{un}$ is Galois
over $X_{{\alpha }}^{un}$.

Then $\left[X; \Omega \right]_{qc} $ is a
directed set, where two elements $Y,Z\in \left[X; \Omega \right]_{qc} $ are
identified with each other if they are isomorphic over $X$ as schemes.

Consider
$
\left[X; \Omega \right]_{et}
$,
the set of finite \'{e}tale Galois covers of X with geometric points
over $s_{\xi}$, where two elements $Y,Z\in \left[X; \Omega \right]_{et} $
are identified with each other if they are isomorphic over $X$ as schemes.

Without loss of generality, in the remainder of the paper,
every $X_{\alpha}\in \left[X; \Omega \right]_{et}$ is
assumed to be essentially affine in the function field $k(X_{\alpha})$ from \emph{Lemma 2.3}.

For any $X_{\alpha},X_{\beta}\in \left[X; \Omega \right]_{et}$, we say $$
X_{\alpha}\leq X_{\beta}$$ if and only if $X_{\beta}$ is a finite \'{e}tale
Galois cover over $X_{\alpha}.$
It follows that $\left[X; \Omega \right]_{et} $ is a
directed set.

It is easily seen that $\left[X; \Omega \right]_{et} $ is a directed subset of $\left[X;
\Omega \right]_{qc} $.

\textbf{Step 6.} Prove $\left[X; \Omega \right]_{et} $ is cofinal in $\left[X; \Omega \right]
_{qc}$.

In deed, according to \emph{Step 5}, we have two inverse systems of
groups
\begin{equation*}
\{Aut(Y/X):Y\in \left[X; \Omega \right]_{et}\};
\end{equation*}
\begin{equation*}
\{Aut(Z/X):Z\in \left[X; \Omega \right]_{qc}\}.
\end{equation*}

Denote by $$\eta:\left[X; \Omega \right]_{et} \hookrightarrow \left[X; \Omega %
\right]_{qc}$$ the inclusion map of the sets. Then $\eta$ induces a map
\begin{equation*}
\eta_{\ast}:\{Aut(Y/X):Y\in \left[X; \Omega \right]_{et}\}\to
\{Aut(Z/X):Z\in \left[X; \Omega \right]_{qc}\}
\end{equation*}
of inverse systems of groups, given by
\begin{equation*}
Aut(Y/X)\mapsto Aut(Y/X).
\end{equation*}
It is clear that $\eta_{\ast}$ is injective.

On the other hand, take any $Z\in \left[X; \Omega \right]_{qc}$. It is seen that $k(Z)$
is a finite Galois extension of $k(X)$. As a subfield of $k(X)^{un}$, $k(Z)$
must be or contained in a finite \'{e}tale Galois extension $k(Z^{\prime})$
of $k(X)$ with $Z^{\prime}\in \left[X; \Omega \right]_{et}$ by \emph{Lemma 5.9} and according to the operations
on arithmetically unramified extensions in \S 5.2.

For $Z^{\prime}$, we have two cases:

\textbf{Case (i)}. Let $k(Z)=k(Z^{\prime})$. Then we have
\begin{equation*}
Aut(Z^{\prime}/X)=Aut(Z/X).
\end{equation*}
In such a case, we put
\begin{equation*}
Z^{\ast}=Z^{\prime}.
\end{equation*}

\textbf{Case (ii)}. Let $k(Z)\subsetneqq k(Z^{\prime})$. Then $k(Z^{\prime})$
is a Galois extension of $k(Z)$ and there is a $Z^{\ast}\in \left[X; \Omega %
\right]_{qc}$ such that
\begin{equation*}
\begin{array}{l}
Z^{\ast}\geq Z; \\
k(Z^{\ast})=k(Z^{\prime}); \\
Aut(Z^{\ast}/X)\cong Gal(k(Z^{\prime})/k(X))\cong Aut(Z^{\prime}/X).
\end{array}%
\end{equation*}

Now define
\begin{equation*}
\left[X; \Omega \right]_{qc}^{\ast}=\{Z^{\ast}:Z\in \left[X; \Omega \right]
_{qc}\}
\end{equation*}
where each $Z^{\ast}$ is given in the manner of the two cases above.

As a directed subset, we have
\begin{equation*}
\left[X; \Omega \right]_{qc}^{\ast}\subseteq \left[X; \Omega \right]_{qc}.
\end{equation*}

It is seen that $\left[X; \Omega \right]_{et} $ is cofinal in $\left[X;
\Omega \right]_{qc}^{\ast}$ and $\left[X; \Omega \right]_{qc}^{\ast}$ is
cofinal in $\left[X; \Omega \right]_{qc}$.
Hence, $\left[X; \Omega \right]_{et} $ is cofinal in $\left[X; \Omega \right]
_{qc}$.

\textbf{Step 7.} As the inverse limits are isomorphic for inverse systems of groups indexed by cofinal directed
sets, we have
\begin{equation*}
\begin{array}{l}
\pi _{1}^{et}\left( X,s\right) \\
={\lim_{\longleftarrow}} _{{Z\in \left[X; \Omega \right]_{et} }}Aut\left(
Z/X\right) \\
\cong \lim_{\longleftarrow _{Z\in \left[X; \Omega \right]_{qc} }}Aut\left(
Z/X\right) \\
\cong \lim_{\longleftarrow _{Z\in \left[X; \Omega \right]_{qc} }}Gal\left(
k\left( Z\right) /k\left( X\right) \right) \\
=\lim_{\longleftarrow _{\alpha \in I}}Gal\left( k\left( X_{\alpha
}^{un}\right) /k\left( X\right) \right) \\
\cong Gal\left( k\left( X_{\infty }^{un}\right) /k\left( X\right) \right) \\
=Gal\left( k\left( X;\Omega\right) ^{un}/k\left( X\right) \right).%
\end{array}%
\end{equation*}
This completes the proof.
\end{proof}

\subsection*{Acknowledgment}

\quad The author would like to express his sincere gratitude to Professor Li
Banghe for his advice and instructions on algebraic geometry and topology.

The author would also like to express his sincere gratitude to Professor
Laurent Moret-Bailly for the stimulating discussion on an earlier version of arithmetically unramified.

The present paper is a revised version of the preprint \cite{an2}.

\end{document}